\documentclass[12pt]{article}
\usepackage[hmargin=1in,vmargin=1in]{geometry} 

\usepackage{verbatim}
\usepackage{amsmath}
 \usepackage{amssymb}
\usepackage{amsthm}
\usepackage{tikz}
\usepackage{graphicx}
\usepackage{color,hyperref}
\definecolor{darkblue}{rgb}{0.0,0.0,0.3}
\hypersetup{colorlinks,breaklinks,
            linkcolor=darkblue,urlcolor=darkblue,
            anchorcolor=darkblue,citecolor=darkblue}

\newtheorem{thm}{Theorem}
\newtheorem{prop}[thm]{Proposition}
\newtheorem{lem}[thm]{Lemma}
\newtheorem{cor}[thm]{Corollary}

\newtheorem{definition}[thm]{Definition}

\newtheorem{ex}[thm]{Example}

\theoremstyle{definition}
\newtheorem{rem}[thm]{Remark}

\newcommand{\N}{\mathbb{N}}

\newcommand{\K}{\mathbb{K}}

\newcommand{\A}{\mathcal{A}}

\newcommand{\D}{\mathcal{D}}
\newcommand{\E}{\mathcal{E}}
\newcommand{\F}{\mathcal{F}}
\renewcommand{\H}{\mathcal{H}}

\newcommand{\Q}{\mathcal{Q}}

\renewcommand{\P}{\mathbb{P}}

\newcommand{\ree}[1]{(\ref{#1})}

\title {Combinatorial Hopf Algebras of  Simplicial Complexes}
\author{Carolina Benedetti,  Joshua Hallam, John Machacek}

\begin{document}

\maketitle
\begin{abstract}
We consider a Hopf algebra of simplicial complexes and provide a cancellation-free formula for its antipode. We then obtain a family of combinatorial Hopf algebras by defining a family of characters on this Hopf algebra.  The characters of these combinatorial Hopf algebras give rise to symmetric functions that encode information about colorings of simplicial complexes and their $f$-vectors. We also use characters to give a generalization of Stanley's $(-1)$-color theorem. A $q$-analogue version of this family of characters is also studied.

\end{abstract}
%
%
\section{Introduction}

As defined in~\cite{abs:chagdse}, a \emph{combinatorial Hopf algebra} is a pair $(\mathcal{H}, \zeta)$ where $\mathcal{H}$ is a graded connected Hopf algebra over some field, $\K$, and $\zeta: \mathcal{H} \to\K$ is an algebra map called a \emph{character of $\mathcal H$}.  Combinatorial Hopf algebras (CHAs) typically  have bases   indexed by  combinatorial objects.  Moreover, characters of a CHA often give rise to enumerative information about these combinatorial objects.

The emerging field of combinatorial Hopf algebras provides an appropriate environment to study subjects with a rich combinatorial structure that originate in other areas of mathematics such as topology, algebra, and geometry. In this paper we study simplicial complexes by endowing them with a Hopf algebra structure. 
This Hopf algebra, which we denote by $\A$, turns out to be a sub Hopf algebra of the hypergraph Hopf algebra studied in~\cite{gsj:gdsrh}. Moreover, the Hopf algebra of graphs $\mathcal G$, studied in \cite{s:iha,hm:ihag,bs:ai}, is a sub Hopf algebra of $\A$. For each positive integer $s$, we define a CHA structure on $\A$.
One of the main results of this paper provides a cancellation-free formula for the antipode of $\A$, and we show how this antipode generalizes the one for $\mathcal G$. 

A beautiful result in~\cite{abs:chagdse} associates a quasisymmetric function to every element in a CHA. This quasisymmetric function often encodes important information about the CHA. In the case of graphs, one can obtain Stanley's chromatic symmetric function. In our case, the quasisymmetric functions that we obtain encode colorings of the simplicial complex as well as its  $f$-vector. We should point out that the $f$-vector is recovered using several of these quasisymmetric functions associated to the simplicial complex. This is one advantage of having a family of CHAs instead using just one.  These quasisymmetric functions give rise to polynomials via principal specializations. They were used in~\cite{dmn:vcsc} to study vertex colorings of simplicial complexes in connection with polynomial identities in the Stanley-Reisner ring associated to the simplicial complex. Here we use the polynomials to give a generalization of Stanley's $(-1)$-color theorem \cite{s:aog} for colorings of simplicial complexes.

The paper is organized as follows. In Section \ref{basics} we review the definitions of combinatorial Hopf algebras and simplicial complexes. We then introduce the Hopf algebra, $\A$, of simplicial complexes and analyze its space of primitive elements.  Section~\ref{anitpodeSec} provides a cancellation-free formula for the antipode of $\A$.  Section~\ref{charAndColorSec} introduces a family of characters $\{\zeta_s\}_{s>0}$ on $\A$ giving rise to families of combinatorial Hopf algebras. Using these characters, we explore the quasisymmetric functions associated to them.   In particular, the power sum expansion of these quasisymmetric functions allows us to recover the $f$-vector. In Section~\ref{evenodd} we provide partial results towards understanding the even and odd subalgebras of $(\A,\zeta_s)$. Finally, in Section~\ref{qanalogue} we define a $q$-analogue of the characters $\zeta_s$, and thus we obtain $q$-analogues for the quasisymmetric functions encoding colorings of simplicial complexes. We then provide an infinite family of unicyclic graphs that can be distinguished using this $q$-analogue, but cannot be distinguished by the chromatic symmetric function. We finish by considering principal specializations  that allow us to obtain certain combinatorial identities.
\section{A Hopf algebra of simplicial complexes}\label{basics}

\subsection{Hopf algebra basics}
We now review some background material on Hopf algebras.  For a more complete study of this topic, the reader is encouraged to see~\cite{dr:hac}.  Let $\mathcal H$ be a vector space over a field $\mathbb K$. Throughout the paper we will assume that $char(\K)=0$. Let $\text{Id}$ be the identity map on $\mathcal H$. We call $\mathcal H$ an associative $\K$-algebra with unit $1$ when $\mathcal H$ is equipped with a $\K$-linear map $m:\mathcal H\otimes\mathcal H\rightarrow\mathcal H$ and an element $1\in\mathcal H$ satisfying
\begin{align*}
m\circ(m\otimes \text{Id}) &= m\circ(\text{Id}\otimes m);\\
m\circ(\text{Id}\otimes u) &= m\circ(u\otimes\text{Id}) = \text{Id}.
\end{align*}
Here, $u$ stands for the $\K$-linear map $ \mathbb K\rightarrow\mathcal H$ defined by $ t\mapsto t\cdot 1$.

A \emph{coalgebra} is a vector space $\D$ over $\mathbb K$  equipped with a coproduct $\Delta:  \D\rightarrow \D\otimes \D$ and a counit $\epsilon:\D\rightarrow \mathbb K$.  Both $\Delta$ and $\epsilon$ must be $\mathbb K$-linear maps. The coproduct  is coassociative so that $(\Delta\otimes \text{Id})\circ\Delta=(\text{Id}\otimes\Delta)\circ\Delta$ and must be compatible with $\epsilon$.  That is,
$$
(\epsilon\otimes \text{Id})\circ\Delta=(\text{Id}\otimes\epsilon)\circ\Delta=\text{Id}.
$$

If an algebra $(\mathcal H,m,u)$ is also equipped with a coalgebra structure given by $\Delta$ and $\epsilon$, then we say that $\mathcal H$ is a \emph{bialgebra} provided $\Delta$ and $\epsilon$ are algebra homomorphisms.

The maps $m$ and $\Delta$ can be applied iteratively as follows. Letting $\mathcal H^{\otimes k}=\mathcal H\otimes\cdots\otimes\mathcal H$ denote the $k$-fold tensor, define the iterated product map $m^{(k-1)}:\mathcal H^{\otimes k}\rightarrow\mathcal H$ inductively by setting $m^{(-1)}=u$, $m^{(0)}=$ \text{Id} and for $k\geq 1$ let $m^{(k)}=m\circ (\text{Id}\otimes m^{(k-1)})$. Similarly, the iterated coproduct map $\Delta^{(k-1)}:\mathcal H\rightarrow \mathcal H^{\otimes k}$ is given inductively by $\Delta^{(k)}=(\text{Id}\otimes \Delta^{(k-1)})\circ\Delta$, where $\Delta^{(-1)}=\epsilon$ and $\Delta^{(0)}=$ Id.

\begin{definition}\label{hopf_algebra}
A \emph{Hopf algebra} $\mathcal H$ is a $\K$-bialgebra together with a $\K$-linear map $S:\mathcal H\rightarrow\mathcal H$ called the antipode. This map must satisfy the following
\begin{equation*}
m\circ(S\otimes \text{Id})\circ\Delta = m\circ(\text{Id}\otimes S)\circ\Delta = u\circ\epsilon.
\end{equation*}
\end{definition}

\begin{rem}
The definition of the antipode given above is rather superficial. The antipode is in fact the inverse of the identity map on $\H$ under the convolution product defined on $\K$-linear maps $f,g:\mathcal H\rightarrow \mathcal{H}$ by $fg:=m\circ(f\otimes g)\circ\Delta$. In fact, given any $\K$-algebra $A$ and $\K$-coalgebra $C$, the convolution product endows the $\K$-linear maps Hom$(C,A)$ with an algebra structure (see \cite[Definition 1.27]{dr:hac}).
\end{rem}

We say that a bialgebra $\mathcal H$ is \emph{graded} if it is decomposed into a direct sum
$$
\mathcal H=\bigoplus_{n\geq 0}H_n
$$
where $m(H_i\otimes H_j)\subseteq H_{i+j}$, $u(\mathbb K)\subseteq H_0$, $\Delta (H_n)\subseteq\bigoplus_{i=0}^nH_i\otimes H_{n-i}$, and  $\epsilon (H_n)=0$ for $n\geq 1$. We call $\mathcal H$ \emph{connected} if $H_0\cong\mathbb K$. For each $n\geq 0$ we refer to elements in $H_n$ as \emph{homogeneous elements of degree $n$.}

Any graded and connected $\K$-bialgebra is a Hopf algebra since the antipode can be defined recursively. In many instances computing the antipode of a given Hopf algebra is a very difficult problem. However, we will provide an explicit cancellation-free formula for the antipode in the Hopf algebra of finite simplicial complexes that we  study here. Now we will introduce some basic concepts about the combinatorial objects we are interested in.

\subsection{Simplicial complexes}

A finite \emph{(abstract) simplicial complex}, $\Gamma$, is a  nonempty collection of subsets of some finite  set $V$ such that $\{v\} \in \Gamma$ for all $v\in V,$ and $X\in \Gamma$ implies $Y\in \Gamma$ for all $Y\subseteq X$.
By convention all our simplicial complexes contain the empty set.
We denote by $\varnothing $ the simplicial complex with empty vertex set.  So $\varnothing$ is the unique simplicial complex whose vertex set is empty and whose only face is the empty set.
The elements of $\Gamma$ are called \emph{faces} and the maximal (with respect to inclusion) faces are called \emph{facets}.
Notice that the facets completely determine the simplicial complex. If $X$ is a face of $\Gamma$ then the  \emph{dimension} of  $X$ is $\dim X = |X|-1$. A face of dimension $s$ is called an \emph{$s$-simplex}. The faces of dimension 0 are called  \emph{vertices} of $\Gamma$ and the set of vertices will be denoted $V(\Gamma)$ where we identify $\{v\}$ with $v$. For instance, if $\Gamma$ has facets $\{1,2,3\}$ and $\{3,4\}$ then $V(\Gamma)=\{1,2,3,4\}$. The dimension of $\Gamma$, written as $\dim \Gamma$,  is the maximum of the dimensions of its facets. 

If $\Gamma$ and $\Theta$ are simplicial complexes with disjoint vertex sets $V_1$ and $V_2$, the \emph{disjoint union} of $\Gamma$ and $\Theta$  is the simplicial complex $\Gamma\uplus\Theta$ with vertex set $V_1\uplus V_2$ and faces $X$ such that  $X\in \Gamma$ or $X\in \Theta$.  If $k$ is a nonnegative integer, the \emph{k-skeleton} of $\Gamma$ is the collection of faces of $\Gamma$ with dimension no greater than $k$.  We will denote the $k$-skeleton of $\Gamma$ by $\Gamma^{(k)}$.   For example, if $\Gamma$ has facets $\{1,2,3\}$ and $\{3,4\}$, then $\Gamma^{(1)}$ is the simplicial complex with facets $\{1,2\},\{1,3\}, \{2,3\}$ and $\{3,4\}$.  
Figure~\ref{OneSkelEx} provides a pictorial representation of this example. Notice that a simple graph gives rise to a simplicial complex of dimension 1 or less. Conversely, a simplicial complex of dimension 1 or less can be thought of as a simple graph.
  
\begin{figure} 
\centering
\includegraphics[width=2.5in]{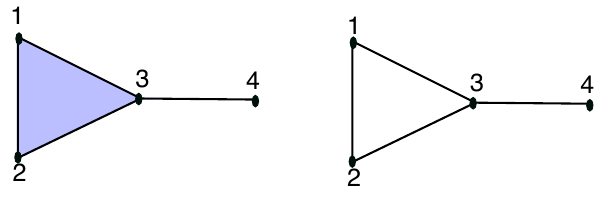}
\caption[OneSkelEx] {A simplicial complex $\Gamma$ and its 1-skeleton $\Gamma^{(1)}$.}\label{OneSkelEx}
\end{figure}  
  
Let $\Gamma$ and $V(\Gamma)$ be defined as above. Given $T\subseteq V(\Gamma)$, define the \emph{induced simplicial complex} of $\Gamma$ on $T$, denoted by $\Gamma_{T}$, to be the simplicial complex  with faces $\{X \cap T \mid X \in \Gamma\}$. So if we return to our example with $\Gamma$ having facets $\{1,2,3\}, \{3,4\}$ and if $T = \{1,3,4\}$, then $\Gamma_T$ has facets $\{1,3\}$ and $\{3,4\}$.

Now we define a Hopf algebra structure on simplicial complexes. Let $\A=\bigoplus_{n\geq 0}A_n$ where $A_n$ is the free $\K$-vector space on the set of isomorphism classes of simplicial complexes on $n$ vertices. Given a simplicial complex, $\Gamma$, we will denote its isomorphism class by $[\Gamma]$.

Define the product $m: \A\otimes \A \rightarrow \A$ by
$$
m\left([\Gamma]\otimes [\Theta]\right) = [\Gamma \uplus \Theta].
$$
Notice that with this multiplication, the unit $u:\K\rightarrow\A$ is given by
$$
u(1)  = [\varnothing].
$$
The coproduct $\Delta: \A\rightarrow \A\otimes \A$, is given by
$$
\Delta([\Gamma]) = \sum_{T \subseteq V(\Gamma)}  [\Gamma _T]\otimes [\Gamma _{V(\Gamma) \setminus T}].
$$
Additionally, define the counit of $\A$ by
$$
\epsilon([\Gamma])  = \delta_{[\Gamma], [\varnothing]}
$$
where $ \delta_{[\Gamma], [\varnothing]}$ is the Kronecker delta.

It follows that $\A$ is a graded, connected $\K$-bialgebra and hence a Hopf algebra. Also, it is not hard to see that $\A$ is commutative and cocommutative. 
From now on, we will drop the brackets from the notation $[\Gamma]$, keeping in mind that we are considering isomorphism classes of simplicial complexes.
In the next section, we turn our attention to the antipode of the Hopf algebra $\A$  and we  provide a cancellation-free formula for it.

\begin{rem}
For the reader interested in the space of primitives of the Hopf algebra $\A$, we remark that a projection onto the space of primitives can be obtained using~\cite[Theorem 10.1]{s:iha}.
\end{rem}

\section{A cancellation-free formula for the antipode}\label{anitpodeSec}

Before stating the main result in this section, we review some basic concepts from graph theory. Suppose $G=(V,E)$ is a graph. A subset $U$ of $V$ is called \emph{stable} if there is no edge between any pair of vertices in $U$.  A \emph{flat}, $F$,  of $G$ is a collection of edges such that in the graph  with vertex set $V$ and edge set $F$, each connected component is an induced subgraph of $G$. If $F$ is a flat then we will denote the subgraph of $G$ with vertex set $V$ and edge set $F$ by $G_{V,F}$ and its number of connected components by $c(F)$.  The set of flats of a graph $G$ will be denoted by $\F(G)$. We denote by $G/F$ the graph obtained from $G$ by contracting the edges in $F$. Recall that an orientation of a graph is called \emph{acyclic} if it does not contain any directed cycles.  The number of acyclic orientations of a graph $G$ will be denoted by $a(G)$. Given an orientation $O$ of $G$, a vertex $v\in V$ is called a \emph{source of $O$} if for every edge $\{v,u\}\in E$, $\{v,u\}$ is oriented away from $v$.

 Let $\Gamma$ be a simplicial complex. Any face $X$ of $\Gamma$ gives rise to a simplicial complex, namely, the simplicial complex formed by all the subsets of $X$. Given a flat $F$ in $\Gamma^{(1)}$ define  $\Gamma_{V,F}$ to be the subcomplex of $\Gamma$, with vertex set $V=V(\Gamma)$, such that $$\Gamma_{V,F}= \{X\in\Gamma\mid X^{(1)}\subset F\}.  $$ 
For example, if we again take $\Gamma$ to have facets $\{1,2,3\},\{3,4\}$ and let $F=\{ \{1,2\},\{1,3\},\{2,3\}\}$ then $\Gamma_{V,F}$ is the simplicial complex with facets $\{1,2,3\},\{4\}$.



In~\cite[Theorem 3.1]{hm:ihag}, the authors provide a cancellation-free formula for the antipode of graphs using induction. Aguiar  and Ardila have also recovered such antipode formula using Hopf monoids~\cite{aa:hmgp}. On the other hand, in \cite[Theorem 7.1]{bs:ai} the authors make use of sign-reversing involutions on combinatorial objects to obtain cancellation-free formulas for antipodes of several Hopf algebras including the graph Hopf algebra. It is worth pointing out that obtaining such a formula is, in general, a difficult problem when studying Hopf algebras arising in combinatorics. However, in our situation, we are able to  use the proof given in  \cite[Theorem 7.1]{bs:ai} to our particular case giving rise to Theorem \ref{anti:thm} below. 

\begin{thm}\label{anti:thm} Let $\Gamma\in A _n$ be a simplicial complex where $n\geq 1$.  Then 
$$S(\Gamma) = \sum_{F \in \F(\Gamma^{(1)})} (-1)^{c(F)} a(\Gamma^{(1)}/F) \Gamma_{V,F}$$
\label{thm:antipode}
where the sum runs over all flats of the 1-skeleton of $\Gamma$.
\end{thm}

\begin{proof}
The strategy to show this result is to define, for every acyclic orientation of each of the graphs $\Gamma^{(1)}/F$, a sign-reversing involution with a unique fixed point. We will only illustrate the proof when $F$ is the empty flat. Namely, we will show in this case that the coefficient of $\Gamma_{V,F}$ equals $a(\Gamma^{(1)})$. At the end of the proof, we will explain how to extend this proof when $F\neq \emptyset$.

Denote by $\mathcal{O}(\Gamma^{(1)})$ the set of acyclic orientations of the graph $\Gamma^{(1)}$ and identify the vertex set $V$ of $\Gamma$ with $[n]:=\{1,2,\dots, n\}$. Using Takeuchi's formula for the antipode in a Hopf algebra (see~\cite{t:fhagc}), given $\Gamma\in A_n$ we obtain
\begin{equation}\label{anti}
S(\Gamma)=\sum_{(V_1,\dots,V_{\ell})\models[n]} (-1)^{\ell}\,\Gamma_{V_1}\uplus\cdots\uplus\Gamma_{V_{\ell}}
\end{equation}
summing over all ordered set partitions $(V_1,\dots,V_{\ell})$ of $[n]$ where all of the $V_i$ are nonempty. Notice that Takeuchi's formula does not provide a cancellation-free expression for the antipode in general.

A term in (\ref{anti}) can be thought as the union of simplicial subcomplexes of $\Gamma$ such that $\Gamma_{V_i}^{(1)}$ is an induced subgraph of $\Gamma^{(1)}$. In particular, notice that when $(V_1,\dots,V_{\ell})$ is such that $\dim(\Gamma_{V_i})=0$ for each $i$, then $\Gamma_{V_1}\uplus\cdots\uplus\Gamma_{V_{\ell}}$ is (isomorphic to) the zero-dimensional subcomplex of $\Gamma$ on the set $[n]$ denoted by $\Gamma_{[n],\emptyset}$.  Note that different ordered set partitions $(V_1,\dots,V_{\ell})\models[n]$ may contribute to the coefficient of $\Gamma_{[n],\emptyset}$ in $(\ref{anti})$. 

 Let $A_{\emptyset}=\{(V_1,\dots,V_{\ell})\models [n] \; | \; \Gamma_{V_1}\uplus\cdots\uplus\Gamma_{V_{\ell}} = \Gamma_{[n],\emptyset}\}$ and define the function
\begin{equation*}
 \rho:A_{\emptyset}\rightarrow \mathcal O(\Gamma^{(1)})
\end{equation*}
 that assigns to $(V_1,\dots,V_{\ell})\in A_{\emptyset}$ an orientation in $\mathcal O(\Gamma^{(1)})$ to each edge $\{i,j\}$ in $\Gamma^{(1)}$ as follows:
\begin{equation*}
 i\rightarrow j\text{\;\;\;\; if \;\;\;\;}i\in V_r, j\in V_s \text{\;\;\;\; and \;\;\;\;}r<s.
 \end{equation*}

Now, given $\sigma=(V_1,\dots,V_{\ell})$ define the \emph{sign} of $\sigma$ to be $sign(\sigma)=(-1)^{\ell}$. Let $O\in\mathcal O(\Gamma^{(1)})$. We will think of $O$ not just as an acyclic orientation but also as the directed graph it induces on the vertex set $[n]$. Such $O$ gives rise to a canonical ordered set partition of $[n]$ in the following manner. Let $v_1$ be the biggest   source of $O_1=O$, then let $v_2$ be the biggest source of $O_2=O_1-\{v_1\}$,  and in general let $v_{k+1}$ be the biggest source of $O_{k+1}=O_k-\{v_k\}$ for $k=0,\dots,n-1$. Then we obtain the ordered set partition $\pi_O=(\{v_1\},\dots, \{v_n\})\models [n]$ such that $v_i$ is the largest source in $O_i$. Since $\rho(\pi_O)=O$, $\rho$ is a surjection. 

For fixed $O\in\mathcal O(\Gamma^{(1)})$ define a sign reversing involution $\iota_O$ on the set $\rho^{-1}(O)$ in the following way.   Set $\iota_O(\pi_O)=\pi_O$. For $\sigma=(V_1,\dots,V_{\ell})\in \rho^{-1}(O)$ such that $\sigma\neq\pi_O$ let $i$ be the smallest index such that $V_i\neq\{ v_i\}$, where $\left(\{v_1\},\dots, \{v_n\}\right)=\pi_O$ as above. The choice of $i$ implies that $v_i\in V_i\cup\cdots\cup V_{\ell}$. Let $V_j$ be the block in $\sigma$ containing $v_i$. If $|V_j|>1$ define
\begin{equation*}
\iota_O(\sigma) = (V_1,\dots,V_{j-1},V_j-\{v_i\},\{v_i\},V_{j+1},\dots,V_{\ell})
\end{equation*}
Otherwise, if $|V_j|=1$ define
\begin{equation*}
\iota_O(\sigma) = (V_1,\dots,V_{j-2},V_{j-1}\cup V_j,V_{j+1},\dots,V_{\ell}).
\end{equation*}
In the latter case, since $v_i$ is the largest source in $O_i$, the vertices in $V_{j-1}$ are vertices in $O_i$ as well and hence, $V_{j-1}\cup V_j$ is a stable set of vertices. Notice that in both cases, $sign(\iota_O(\sigma))=-sign(\sigma)$. Moreover, $\iota_O(\iota_O(\sigma))=\sigma$ and $\pi_O$ is the unique fixed point of $\iota_O$.
We conclude that for each acyclic orientation $O$, the involution $\iota_O$ has a unique fixed point. Hence the coefficient of $\Gamma_{[n],\emptyset}$ in (\ref{anti}) is $(-1)^na(\Gamma^{(1)})$.

The proof for the coefficient of $\Gamma_{V,F}$ when $F\neq\emptyset$ can be done using the same argument as above with slight modifications. Namely, each connected component of $\Gamma_{V,F}$ can be identified with a single vertex and a similar sign reversing involution can be defined for the graph $\Gamma^{(1)}/F$ whose vertex set has cardinality $c(F)$. \end{proof}

Let us return to our previous example with $\Gamma$ generated by  the facets $\{1,2,3\}$ and $ \{3,4\}$. Using the information in Table~\ref{antipodeTab} we obtain the expression in Figure \ref{antipode}. Looking at the expression for the antipode in this example, we see that if we  add all the coefficients together we obtain $1$. It turns out that the sum of the coefficients of the antipode of a simplicial complex is always $(-1)^n$ where $n$ is the number of vertices of the simplicial complex.  We will derive this fact using characters and quasisymmetric functions in the next section (see Corollary~\ref{sumIsNeg1Cor}).

\begin{table}
\begin{center}
  \begin{tabular}{| l | c | c |}
 \hline
   $F\in \F(\Gamma^{(1)})$ & $(-1)^{c(F)}$ & $a(\Gamma^{(1)}/F)$\\ \hline
    $\emptyset$ & $(-1)^4$ & 12 \\ \hline
    $\{1,2\}$ & $(-1)^3$  & 4 \\\hline
    $\{1,3\}$ & $(-1)^3$  & 4 \\ \hline
    $\{2,3\}$ & $(-1)^3$  & 4 \\ \hline
    $\{3,4\}$ & $(-1)^3$  & 6 \\ \hline
    $\{1,2\},\{3,4\}$ & $(-1)^2$  &2 \\ \hline
    $\{1,3\},\{3,4\}$ & $(-1)^2$  &2 \\ \hline
    $\{2,3\},\{3,4\}$ & $(-1)^2$  &2 \\ \hline
    $\{1,2\},\{1,3\},\{2,3\}$ & $(-1)^2$  &2 \\ \hline
    $\{1,2\},\{1,3\},\{2,3\},\{3,4\}$ & $(-1)^1$  & 1 \\  
\hline
  \end{tabular}
\caption{Information to compute the antipode of $\Gamma$.} \label{antipodeTab}
\end{center}
\end{table}


 \begin{figure}
 \centering
 \includegraphics[width=4.5in]{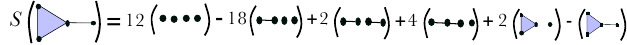}
 \caption[antipode]{Antipode of an element in $A_4$.}
\label{antipode}
\end{figure}

Now, note that once we have computed $S(\Gamma)$, we can easily find the antipode of the simplicial complex $\Gamma^{(1)}$ by just taking the 1-skeleton of each of the terms in the sum for the antipode.  So we immediately get that 
 $$
S(\Gamma^{(1)})=12 \overline{K_4}-18 (K_2\uplus \overline{K_2})+2(K_2\uplus K_2) + 4(P_3\uplus K_1)+ 2(K_3\uplus  K_1)-\Gamma^{(1)}.
$$
 where $K_n$ is the complete graph on $n$ vertices, $\overline{K_n}$ is the complement of the complete graph on $n$ vertices, and $P_n$ is the path on $n$ vertices. 

More generally, let $\A^{(k)}$ be the $\K$-linear span of isomorphism classes of simplicial complexes of dimension at most $k$.  That is, complexes $\Gamma\in\A$ such that $\Gamma^{(k)} = \Gamma$.
For each $k\geq 0$, we define the map
\begin{align*}
\phi_k:\;&\A\rightarrow\A ^{(k)}\\
&\;\Gamma\mapsto\Gamma^{(k)}
\end{align*}
 which takes the $k$-skeleton of a simplicial complex.
We extend this map linearly to all of $\A$.

\begin{prop}\label{prop:subalg}For any nonnegative integer $k$, $\A^{(k)}$ is a Hopf subalgebra of $\A$ and the map $\phi_k:\A \to \A^{(k)}$ is a Hopf algebra homomorphism.
\end{prop}
\begin{proof}
Let $\Gamma$ and $\Theta$ be simplicial complexes. Since $\dim \Gamma \uplus \Theta = \max \{ \dim \Gamma, \dim \Theta\}$  and  $\dim \Gamma_T \leq \dim \Gamma$ for any $T \subseteq V(\Gamma)$ it follows that $\A^{(k)}$ is a Hopf subalgebra.
Observe that
$$
(\Gamma \uplus \Theta)^{(k)} = \{X : X \in \Gamma \uplus \Theta, |X| \leq k+1\} = \{X \in \Gamma: |X|\leq k+1\} \cup \{X \in \Theta: |X| \leq k+1\}.
$$
Therefore $(\Gamma \uplus \Theta)^{(k)} =\Gamma^{(k)} \uplus \Theta^{(k)}$ and $\phi_k$ is an algebra homomorphism.
Next, since $$ (\Gamma_T)^{(k)}  = \{X \in \Gamma: X \subseteq T, |X| \leq k+1\} = (\Gamma^{(k)})_T$$ we have
$$\sum_{T \subseteq V(\Gamma)} (\Gamma_T)^{(k)} \otimes (\Gamma_{V(\Gamma) - T})^{(k)} = \sum_{T \subseteq V(\Gamma)} (\Gamma^{(k)})_T \otimes (\Gamma^{(k)})_{V(\Gamma) - T}$$
and so $\phi_k$ is also a coalgebra homomorphism. We conclude that  $\phi_k$ is a Hopf algebra homomorphism.
\end{proof}

Using Proposition~\ref{prop:subalg} along with the fact that for any Hopf algebra homomorphism $\beta: \H_1 \to \H_2$ one has $\beta(S_{\H_1}(h)) = S_{\H_2}(\beta(h))$ for all $h \in \H_1$  (see \cite[Proposition 1.46]{dr:hac}) we can conclude that $S \circ \phi_k = \phi_k \circ S$.
This means that if $k$ is a nonnegative integer  and $S(\Gamma) = \sum c_i \Gamma_i$, then $S(\Gamma^{(k)})=\sum c_i\Gamma_i^{(k)}$.

\section{Characters and quasisymmetric functions}\label{charAndColorSec}

Now that we have endowed $\A$ with a Hopf algebra structure, we will proceed to define a family of characters on $\A$. This will give rise to a family of combinatorial Hopf algebras. We will then show how these characters give combinatorial information about simplicial complexes.

\subsection{The Hopf algebra $\Q Sym$}
We review some key facts about characters and quasisymmetric functions.
More details can be found in~\cite{abs:chagdse}. The Hopf algebra of quasisymmetric functions $\Q Sym$ is graded as $\Q Sym = \bigoplus_{n \geq 0} \Q Sym_n$ where $\Q Sym_n$ is spanned linearly over $\K$ by $\{M_{\alpha}\}_{\alpha \vDash n}$. Here $M_{\alpha}$ is defined by
$$
M_{\alpha} := \sum_{i_1 < i_2 < \cdots < i_l} x_{i_1}^{\alpha_1} x_{i_2}^{\alpha_2} \cdots x_{i_\ell}^{\alpha_\ell}
$$
where $\alpha = (\alpha_1,\dots,\alpha_{\ell})$ is a composition of $n$. The basis given by $\{M_{\alpha}\}$ is known as the \emph{monomial basis of $\Q Sym$}.
We have $M_{()} = 1$,  which spans $\Q Sym_0$, where $()$ is  the composition of 0 with no parts.

Let the map $\zeta_{\Q}: \Q Sym \rightarrow \K$  be defined as $\zeta_{\Q}(f) = f(1,0,0,\dots)$ for a quasisymmetric function  $f(x_1,x_2,x_3,\dots)$. Given that $\zeta_{\Q}$ is an evaluation map, it is also an algebra map and hence a character of $\Q Sym$. This endows $\Q Sym$ with a combinatorial Hopf algebra structure. Moreover, Theorem 4.1 of~\cite{abs:chagdse} states that given a combinatorial Hopf algebra $(\H, \zeta)$ there is a unique combinatorial Hopf algebra homomorphism
$$\Psi_{\zeta}: \H \to \Q Sym$$
given by
\begin{equation}
\Psi_{\zeta}(h) = \sum_{\alpha=(\alpha_1,\dots, \alpha_\ell) \vDash n} \zeta_{\alpha}(h) M_{\alpha}
\label{eq:Psi}
\end{equation}
 for $h$ homogeneous of degree $n$, where $\zeta_{\alpha}$ is the composition of functions
$$\H \xrightarrow{\Delta^{(\ell-1)}} \H^{\otimes \ell} \longrightarrow H_{\alpha_1} \otimes H_{\alpha_2} \otimes \cdots \otimes H_{\alpha_\ell} \xrightarrow{\zeta^{\otimes \ell}} \K.$$
Here the unlabeled map is the canonical projection and $\alpha = (\alpha_1, \alpha_2, \dots, \alpha_\ell)$. 


Now, for each $s > 0$, define the map $\zeta_s:\A\rightarrow \mathbb K$ by  $$\zeta_s(\Gamma) = \begin{cases} 1 & \dim \Gamma < s, \\ 0 & \dim \Gamma \geq s, \end{cases}$$ and extend linearly to $\A$. Each map $\zeta_s$ is multiplicative, i.e. $\zeta_s(\Gamma\uplus\Theta)=\zeta_s(\Gamma)\zeta_s(\Theta)$. Thus, for each $s$ the pair $(\A,\zeta_s)$ is a combinatorial Hopf algebra. Moreover, since $\A$ is cocommutative, equation~(\ref{eq:Psi}) implies that $\Psi_{\zeta}$ is actually a symmetric function. In particular, 
\begin{equation}
\Psi_{\zeta}(h) = \sum_{\lambda \vdash n} \zeta_{\lambda}(h) m_{\lambda} \;\;\;\;\;\;\;\;\text{where}\;\;\;\;\;\;\;m_{\lambda}=\sum_{\alpha}{M_{\alpha}}
\label{eq:monomial}
\end{equation}
summing over all the compositions $\alpha$ that can be rearranged to the partition $\lambda$. For instance, the compositions $(1,2)$ and $(2,1)$ rearrange to the partition $(2,1)$.

%
%

Next we will review some concepts concerning colorings of simplicial complexes. This will allow us to connect the quasisymmetric functions  associated to $\A$ with such colorings.

\subsection{Colorings}

Let $\mathbb P$ denote the set of positive integers and let $G$ be a graph with vertex set $V(G)$. A \emph{coloring} of $G$ is a map $f:V(G)\rightarrow\P$. We refer to $f(u)$ as the color of $u$.
A \emph{proper coloring} of $V$ is a coloring such that $f(u)\neq f(v)$ whenever $uv$ is an edge of $G$.  Given a simplicial  complex $\Gamma$ and $s\in\N$, define an $s$\emph{-simplicial coloring}\footnote{In ~\cite{dmn:vcsc} the authors use the term $(P,s)$-coloring for an $s$\emph{-simplicial coloring} which uses some palette of colors $P \subseteq \P$. To avoid confusion with terminology in graphs, we have adopted the term $s$\emph{-simplicial coloring.}} to be a coloring of $V(\Gamma)$ such that there is no monochromatic face of dimension $s$.
Notice that any 1-simplicial coloring of $\Gamma$ is simply a proper coloring of its 1-skeleton $\Gamma^{(1)}$. In Figure~\ref{coloringex} we use our earlier example and depict two colorings of $\Gamma$ using the colors $\{x,y,z\}\subseteq\P$.  
 
\begin{figure}
\begin{center}
\includegraphics[width=2.5in]{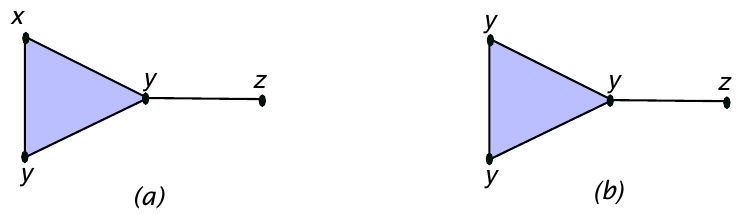}
\end{center}
\caption{A 2-coloring of $\Gamma$ in \emph{(a)} and a 3-coloring of $\Gamma$ in $(b)$.} \label{coloringex}
\end{figure}

Given a graph $G$, the number of proper colorings $f:V(G) \to \{1,2,\dots,t\}$ is the well-known chromatic polynomial, $\chi(G; t)$.
For a simplicial complex $\Gamma$ the number of $s$-simplicial colorings $f: V(\Gamma) \to \{1,2,\dots,t\}$ is called the \emph{$s$-chromatic polynomial}, $\chi_s(\Gamma;t)$, and defined in~\cite{n:scp,MN}. Although it is not obvious that $\chi_s(\Gamma;t)$ is a polynomial, we will see that this is the case once we realize it as a principal specialization of a certain symmetric function.

Stanley provided a generalization (see~\cite{s:sfgcpg}) of the chromatic polynomial of a graph $G$ by defining
$$
\psi(G; x_1, x_2, \dots) = \sum_{f} \prod_{i \geq 1} x_i^{|f^{-1}(i)|}
$$
where the sum is over proper colorings $f: V \to \mathbb P$. This formal power series is known as \emph{Stanley's chromatic symmetric function}.
For a simplicial complex $\Gamma$ we define the \emph{$s$-chromatic symmetric function} as
$$\psi_s(\Gamma; x_1, x_2, \dots) = \sum_{f}\prod_{i \geq 1} x_i^{|f^{-1}(i)|}$$
where now the sum is over  $s$-simplicial colorings $f: V \to \mathbb P$ and $V=V(\Gamma)$. Notice that when $s=1$ we obtain Stanley's chromatic symmetric function. Given $f(x_1,x_2,\dots)\in\Q Sym$ its \emph{principal specialization at $t$} is defined by $ps^{1}(f)(t)=f(1,\dots,1,0,0,\dots)$ where $t\in\mathbb P$ and only the first $t$ variables are specialized to 1. It turns out that $ps^{1}(f)(t)$ gives rise to a \emph{unique} polynomial in $t$ (see \cite[Proposition 7.7]{dr:hac}). The $s$-chromatic polynomial $\chi_s(\Gamma;t)$ is the polynomial determined $ps^{1}(\psi_s(\Gamma))(t)$. 
We now show how the $s$-chromatic symmetric function arises from the CHA $(\A, \zeta_s)$.

\begin{thm}
Fix $s$ and consider the combinatorial Hopf algebra $(\A, \zeta_s)$.  If $\Gamma$ is a simplicial complex, then  $\Psi_{\zeta_s}(\Gamma) = \psi_s(\Gamma; x_1, x_2, \dots)$.

\end{thm}
\begin{proof}
Consider the  formula in equation~(\ref{eq:Psi}).
Given a simplicial complex $\Gamma \in A_n$ and a composition $\alpha =( \alpha_1, \alpha_2, \dots, \alpha_\ell) \vDash n$, we get that the coefficient of $M_{\alpha}$ is the number of ordered set partitions $V_1 \uplus V_2 \uplus \cdots \uplus V_\ell$ of $V(\Gamma)$ such that $|V_i| = \alpha_i$ and $\dim \Gamma_{V_i} < s$ for each $i$.
In an $s$-simplicial coloring, every element of a subset $T$ of $V(\Gamma)$ can be assigned the same color if and only if $\dim \Gamma_T<s$.
Thus the coefficient of $M_{\alpha}$ counts $s$-simplicial colorings using only colors $\{j_1 < j_2 < \cdots < j_\ell\} \subseteq \P$ where $|f^{-1}(j_i)| = \alpha_i$ for each $i$.
The result follows.
\end{proof}

We discuss now the expansion of the symmetric function $\Psi_{\zeta_s}(\Gamma; x_1, x_2, \dots)$ in terms of the \emph{power sum basis}. The \emph{power sum symmetric function} of degree $n$, denoted by $p_{n}$, in the variables $x_1,x_2,\dots$ is given by $p_n=\sum_{i\geq 1}x_i^n$ and for $\lambda=(\lambda_1,\dots,\lambda_{\ell})$ an integer partition define 
$$
p_{\lambda}:=p_{\lambda_1}\cdots  p_{\lambda_{\ell}}.
$$

Take a simplicial complex $\Gamma$ and let $V = V(\Gamma)$.
For any $s>0$, we denote the collection of $s$-simplices of $\Gamma$ by $F_s(\Gamma)$.
Given any $A \subseteq F_s(\Gamma)$ we let $\Gamma_{V,A}$ be the simplicial complex on the vertex set $V$ generated by $A$.
That is, the faces of $\Gamma_{V,A}$ of dimension greater than 0 are subsets $X \subseteq V$ such that $X \subseteq Y$ for some $Y \in A$.
For $A \subseteq F_s(\Gamma)$ we define an integer partition $\lambda(A)$  which has length the number of connected components of $(\Gamma_{V,A})^{(1)}$ and whose parts are given by the number of vertices in each connected component. The $s$-chromatic symmetric function has the following expansion in the power sum basis
\begin{equation}
\Psi_{\zeta_s}(\Gamma) = \sum_{A \subseteq F_s(\Gamma)} (-1)^{|A|} p_{\lambda(A)}
\label{eq:p}
\end{equation}
which can be shown analogously to~\cite[Theorem 2.5]{s:sfgcpg}.


 Let $\Gamma$ denote the  $(n-1)$-simplex, i.e., the simplicial complex on $[n]$ whose only facet is the set $[n]$ itself.  We now look at the monomial and Schur expansions of $\Psi_{\zeta_s}(\Gamma)$.  We have

$$
\Psi_{\zeta_s}(\Gamma) = \sum_{\mu\vdash n\atop\mu_1\leq s}{n\choose{\mu_1,\cdots,\mu_{\ell}}}m_{\mu}
$$
where, for every $s\geq 1$, the sum is over partitions $\mu=(\mu_1,\cdots,\mu_{\ell})$ of $n$ such that $\mu_1\leq s$.


 In the above case when $s=n$ we get
$$
\Psi_{\zeta_n}(\Gamma) = \sum_{\lambda\vdash n}f^{\lambda}s_{\lambda}
$$
where $f^{\lambda}$ is the number of standard Young tableaux of shape $\lambda$. This follows since
\begin{equation*}
\sum_{\mu\vdash n}{n\choose{\mu_1,\cdots,\mu_{\ell}}}m_{\mu}= (m_{(1)})^n= (s_{(1)})^n =   \sum_{\lambda\vdash n}f^{\lambda}s_{\lambda}.  
\end{equation*}
Moreover, since $m_{(n)}=p_{n}=\sum_{i=0}^n(-1)^{i}s_{(n-i,1^i)}$ we conclude that the symmetric function 
$$
\Psi_{\zeta_{n-1}}(\Gamma) = \sum_{\lambda\vdash n}f^{\lambda}s_{\lambda}-m_{(n)}
$$
is Schur positive as well. Unfortunately, the functions $\Psi_{\zeta_s}(\Gamma) $ are not always Schur positive. An instance of this is when $n=4$ and $s=2$.  In this case,
\begin{align*}
\Psi_{\zeta_s}(\Gamma) &= 6m_{(2,2)}+12m_{(2,1,1)}+24m_{(1,1,1,1)}\\
&= 6s_{(2,2)} + 6s_{(2,1,1)}-6s_{(1,1,1,1)}.
\end{align*}
It would be interesting to determine other families of simplicial complexes that give Schur positivity of the functions $\Psi_{\zeta_s}$ for different values of $s$.

\subsection{Acyclic orientations and chromatic polynomial evaluations}

In this subsection, we use our antipode formula along with the characters defined above to interpret certain evaluations of  the $s$-chromatic polynomial. Given any character $\zeta:\A\rightarrow\mathbb K$, the following identity holds (see \cite[Section 1]{abs:chagdse})
$$
\zeta^{-1}=\zeta\circ S
$$
where $S$ is the antipode in $\A$ and $\zeta^{-1}$ is the inverse of $\zeta$ under convolution. In other words, 
$\zeta^{-1}\zeta=u\circ\epsilon
$ where $\zeta^{-1}\zeta=m\circ(\zeta^{-1}\otimes\zeta)\circ\Delta$.

Now, since $ps^{1}(\psi_s(\Gamma))(t)=\chi_s(\Gamma;t)$, using ~\cite[Proposition 7.7 (iii)]{dr:hac} yields
\begin{equation}\label{eval:minus1}
\zeta_s\circ S(\Gamma)=\zeta_s^{-1}(\Gamma)=ps^{1}(\psi_s(\Gamma))(-1)=\chi_s(\Gamma;-1).
\end{equation}
This allows us to prove the following theorem.
\begin{thm}\label{genacyclic}
Let $\Gamma \in A_n$ be a simplicial complex and let $s$ be a positive integer.  Then 
$$\chi_s(\Gamma;-1) = \sum_{\substack{F \in \F(\Gamma^{(1)}) \\ \dim \Gamma_{V,F} < s}} (-1)^{c(F)}a(\Gamma^{(1)}/F).$$
\label{prop:-1}
\end{thm}
\begin{proof}
Using equation (\ref{eval:minus1}), the fact that $\zeta_s^{-1}=\zeta_s\circ S$, and our antipode formula in Theorem~\ref{anti:thm} yields
\begin{align*}
\chi_s(\Gamma;-1) &= \zeta_s(S(\Gamma))\\
&= \sum_{F \in \F(\Gamma^{(1)})} (-1)^{c(F)} a(\Gamma^{(1)}/F) \zeta_s(\Gamma_{V,F})\\
&= \sum_{\substack{F \in \F(\Gamma^{(1)}) \\ \dim \Gamma_{V,F} < s}} (-1)^{c(F)}a(\Gamma^{(1)}/F)
\end{align*}
and so the result is proven.
\end{proof}

This result shows that like the chromatic polynomial for graphs, the evaluation at $t=-1$ of the $s$-chromatic polynomial for simplicial complexes has a combinatorial interpretation in terms of counting acyclic orientations. If we let $s=1$  in Theorem~\ref{prop:-1}, we  recover Stanley's classical result~\cite{s:aog} that $\chi_1(\Gamma; -1) = (-1)^n a(\Gamma^{(1)})$.  In~\cite[Example 3.3]{hm:ihag} the authors perform  the same calculation with characters for the Hopf algebra of graphs.  In addition to the previous result, we also get the following corollary.

\begin{cor}\label{sumIsNeg1Cor} Let $\Gamma$ be a simplicial complex on $n$ vertices,  then we have the following
$$
 (-1)^n =  \sum_{F \in \F(\Gamma^{(1)})} (-1)^{c(F)} a(\Gamma^{(1)}/F).
$$
\end{cor}
\begin{proof}

If we take $s > \dim \Gamma$, then $\chi_s(\Gamma; t) = t^n$ since there is no restriction on coloring.
So, $\chi_s(\Gamma; -1) = (-1)^n$. 
Meanwhile, the sum in Theorem~\ref{prop:-1} runs over all $F \in \F(\Gamma^{(1)})$ because the condition $\dim \Gamma_{V,F} < s$ is always satisfied.
\end{proof}

\subsection{The $f$-vector}\label{fvector}

 Given a simplicial complex $\Gamma$, the \emph{$f$-vector} of $\Gamma$ is  defined to be $(f_0,f_1,\dots )$ where $f_s$ is the number of $s$-simplices in $\Gamma$.  
 For example, if $\Gamma$ is the simplicial complex generated by the facets $\{1,2,3\}$ and  $\{3,4\}$, then $\Gamma$ has $f$-vector $(4,4,1,0,0,\dots)$.
In this section we show how to obtain the $f$-vector of a simplicial complex from the symmetric functions $\{\Psi_{\zeta_s}\}_{s>0}$.


Let  $[p_{\lambda}] \Psi_{\zeta_s}(\Gamma)$ denote the coefficient of $p_{\lambda}$ in the power sum expansion of $\Psi_{\zeta_s}(\Gamma)$.
If $\Gamma$ is a simplicial complex on $n$ vertices and $A \subseteq F_s(\Gamma)$, then $\lambda(A) = (s+1,1^{n-s-1})$ if only if $A$ consists of a single $s$-simplex.
By considering equation~(\ref{eq:p}) we obtain the following proposition.

\begin{prop}
If $\Gamma$ is a simplicial complex with $|V(\Gamma)| = n$ and $s>0$, then
$$
f_{s} = -[p_{(s+1,1^{n-s-1})}]\Psi_{\zeta_s}(\Gamma).
$$
\label{prop:f}
\end{prop}



Given a simplicial complex $\Gamma$, denote its $s^{th}$ homology group  by $H_s(\Gamma)$ for each $s \geq 0$.
The $s^{th}$ Betti number is denoted $\beta_s(\Gamma)$ and defined to be the rank of $H_s(\Gamma)$.
One useful fact about homology groups is that if $\dim \Gamma = k$, then $H_s(\Gamma) = 0$ for $s > k$.
In particular, this means $\beta_s(\Gamma) = 0$ for $s > k$. 

Note that  Proposition~\ref{prop:f} allows us to recover the Euler characteristic $\chi_{\Gamma}$ of $\Gamma$, since $\chi_{\Gamma}=\sum_{s \geq 0}(-1)^s f_s=\sum_{s \geq 0} (-1)^s \beta_s$ where  $\beta_s = \beta_s(\Gamma)$.
 Since we can determine the $f$-vector from the $s$-chromatic symmetric functions, it is natural to wonder if we  can also determine the Betti numbers. If $\Gamma$ is a graph, i.e. if $\dim(\Gamma)\leq 1$, then $\beta_{0}$ equals the number of its connected components. This number can also be recovered by means of the chromatic polynomial of $\Gamma$. Thus, in this case we recover the sequence of Betti numbers $(\beta_0, \beta_1,0,0,...)$.  However, for higher dimensional simplicial complexes this is not always the case as we see in the next example.

%
%
%
%
%

\begin{ex} We now consider two simplicial complexes $\Gamma$ and $\Theta$ such that $\Psi_{\zeta_s}(\Gamma) = \Psi_{\zeta_s}(\Theta)$ for all $s > 0$, but $\Gamma$ and $\Theta$ have different Betti numbers.
We set $\Gamma = X \uplus Y$ and $\Theta = Z \uplus W$ where $X$, $Y$, $Z$, and $W$ are given in Table~\ref{XYZWTab}.
The simplicial complexes $Y$ and $Z$ are shown Figure~\ref{betti_ex}.

\begin{figure}
\begin{center}
\includegraphics[width=2.5in]{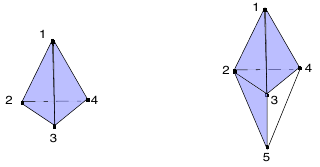}
\end{center}
\caption{\small{$Z$ (left) has 2-simplicies $123, 124, 134, 234$ and $Y$ (right)  has 2-simplices $123,124,134,235$.}} \label{betti_ex}
\end{figure}


Since $\Gamma^{(1)}=\Theta^{(1)}$, it follows $\Psi_{\zeta_1}(\Gamma) = \Psi_{\zeta_1}(\Theta)$. Also,
$$\Psi_{\zeta_2}(\Gamma) = (p_{(1^4)})(p_{(1^5)} - 4p_{(3,1^2)} + 3p_{(4,1)}) = p_{(1^9)} - 4p_{(3,1^6)} + 3p_{(4,1^5)}$$
$$\Psi_{\zeta_2}(\Theta) = (p_{(1^4)} - 4p_{(3,1)} + 3p_{(4)})(p_{(1^5)}) = p_{(1^9)} - 4p_{(3,1^6)} + 3p_{(4,1^5)}$$
and  so $\Gamma$ and $\Theta$ have the same $2$-chromatic symmetric function.

\begin{table}
\begin{center}
  \begin{tabular}{| l | c | c |}
 \hline
    & Vertices & Facets\\ \hline
    $X$ & $1,2,3,4$ & $12,13,14,23,24,34$\\ \hline
    $Y$ & $1,2,3,4,5$  &  $123,124, 134, 235, 45$\\ \hline
    $Z$ & $1,2,3,4$  & $123,124,134,234$ \\ \hline
    $W$ & $1,2,3,4,5$ & $12,13,14,23,24,25,34,35,45$\\
\hline
  \end{tabular}
\caption{Vertices and facets of the complexes $X, Y, Z,$ and $W.$} \label{XYZWTab}
\end{center}
\end{table}

Since $\Gamma$ and $\Theta$ are both 2-dimensional simplicial complexes,  we conclude $\Psi_{\zeta_s}(\Gamma)\ = \Psi_{\zeta_s}(\Theta)$ for all $s > 0$.
However, the Betti numbers of $\Gamma$ and $\Theta$ are not the same since $\beta_2(\Theta) = 1$ while $\beta_2(\Gamma) = 0$.

\end{ex}

\section{The even and odd subalgebras}\label{evenodd}

For a CHA $(\H, \zeta)$, where $\H = \bigoplus_{n\geq 0 } H_n$, the even and odd subalgebras, denoted by $S_{+}(\H, \zeta)$ and  $S_{-}(\H, \zeta)$ respectively, were originally defined in~\cite{abs:chagdse}.
Let $\bar{\zeta}$ denote the character defined by $\bar{\zeta}(h) = (-1)^n \zeta(h)$ for a homogenous element $h \in H_n$.
Let $h \in H_n$ be any homogenous element, then $h \in S_{+}(\H, \zeta)$ if and only if either of the following equivalent conditions is satisfied 
\begin{align}
(\text{Id} \otimes (\bar{\zeta} - \zeta) \otimes \text{Id}) \circ \Delta^{(2)}(h) &= 0 &(\text{Id} \otimes (\bar{\zeta}^{-1} \zeta - \epsilon) \otimes \text{Id}) \circ \Delta^{(2)}(h) &= 0.
\label{eq:even}
\end{align}
Recall, products of characters like $\bar{\zeta}^{-1} \zeta$ refer to the convolution product.
Similarly, for a homogeneous element $h \in \H$ we have $h \in S_{-}(\H, \zeta)$ if and only if either of the following equivalent conditions is satisfied 
\begin{align}
(\text{Id} \otimes (\bar{\zeta} - \zeta^{-1}) \otimes \text{Id}) \circ \Delta^{(2)}(h) &= 0 &(\text{Id} \otimes (\bar{\zeta} \zeta - \epsilon) \otimes \text{Id}) \circ \Delta^{(2)}(h) &= 0.
\label{eq:DehnSommerville}
\end{align}
Either equation in~(\ref{eq:DehnSommerville}) is called the generalized Dehn-Sommerville relations for the CHA $(\H, \zeta)$.

Let us analyze the even subalgebra of $(\A,\zeta_s)$ for any $s$.
The equation in the left in~(\ref{eq:even}) involves $\bar{\zeta_s} - \zeta_s$ which is always nonpositive.
This means no cancellation can occur.
So, for a simplicial complex $\Gamma$ this implies $\Gamma \in S_+(\A, \zeta_s)$ if and only if $(\bar{\zeta_s} - \zeta_s)(\Gamma_A) = 0$ for all $A \subseteq V(\Gamma)$.
It follows that $\Gamma  \in S_+(\A, \zeta_s)$ if and only if $\Gamma = \varnothing$.

Classifying when a simplicial complex is in the odd subalgebra is more difficult because the equations in~(\ref{eq:DehnSommerville}) can contain both positive and negative terms. Let $\E(\A, \zeta_s)$ denote the $\K$-space given by 
$$\E(\A, \zeta_s)=\text{span}_{\K}\{\Gamma\; | \; (\bar{\zeta_s} \zeta_s)(\Gamma_A) = \epsilon(\Gamma_A) \text{ for all }A\subseteq V(\Gamma)\}.$$
 Observe that if $\Gamma\in\E(\A, \zeta_s)$ then $\Gamma\in S_-(\A, \zeta_s).$ This can be checked from (\ref{eq:DehnSommerville}). 
 Now we provide some lemmas that will allow us to describe some of the elements in $\E(\A, \zeta_s)$, and hence, in $S_-(\A, \zeta_s)$. 



\begin{lem}
If $\Gamma$ is a simplicial complex and $\dim \Gamma < s$, then $\Gamma \in \E(\A, \zeta_s)$. 
\label{lem:dim}
\end{lem}

\begin{proof}
If $\dim \Gamma < s$,  then 
\begin{align*}
(\bar{\zeta_s} \zeta_s)(\Gamma) &= \sum_{A \uplus B = V(\Gamma)} \bar{\zeta_s}(\Gamma_A) \zeta_s(\Gamma_B) \\
&= \sum_{A \subseteq V(\Gamma)} (-1)^{|A|}  \\
&= \delta_{\Gamma, \varnothing}\\
&= \epsilon(\Gamma).
\end{align*}
When $\dim \Gamma < s$, then also $\dim \Gamma_A < s$ for any $A \subseteq V(\Gamma)$.
Thus $\Gamma \in \E(\A, \zeta_s)$.
\end{proof}

\begin{lem}
If $\Gamma$ is the $s$-simplex, then $(\bar{\zeta_s} \zeta_s)(\Gamma) = -(1+(-1)^{s+1})$.
\label{lem:simplex}
\end{lem}

\begin{proof}
We compute for $\Gamma$ the $s$-simplex
\begin{align*}
(\bar{\zeta_s} \zeta_s)(\Gamma) &= \sum_{A \uplus B = V(\Gamma)} \bar{\zeta_s}(\Gamma_A) \zeta_s(\Gamma_B) \\
&= \sum_{k=1}^s (-1)^k \binom{s+1}{k}\\
&= -(1 + (-1)^{s+1}).
\end{align*}
\end{proof}


Lemma~\ref{lem:dim} shows that $\A^{(s-1)}$ is contained is $\E(\A, \zeta_s)$ for every $s$.
Recall, $\A^{(k)}$ is the Hopf subalgebra spanned by simplicial complexes of dimension at most $k$ from Proposition~\ref{prop:subalg}.
Lemma~\ref{lem:simplex} implies that $\E(\A, \zeta_s) = \A^{(s-1)}$ when $s$ is odd.
However, the $s$-simplex is an element of $\E(\A, \zeta_s)$ when $s$ is even.



\begin{prop}
If $s$ is odd and $\Gamma$ is a simplicial complex, then $\Gamma \in S_{-}(\A, \zeta_s)$ if and only if $\dim \Gamma < s$.
\end{prop}
\begin{proof}
If $\dim \Gamma < s$, then $\Gamma \in S_{-}(\A, \zeta_s)$ by Lemma~\ref{lem:dim}.
So, it remains to show if $\dim \Gamma \geq s$, then $\Gamma \not\in S_{-}(\A, \zeta_s)$ for $s$ odd.
Recall that $\Gamma \in S_{-}(\A, \zeta_s)$ if and only if
$$\sum_{A \uplus B \uplus C = V(\Gamma)} \Gamma_A \otimes (\bar{\zeta_s} \zeta_s - \epsilon)(\Gamma_B) \otimes \Gamma_C = 0.$$
It follows that $\Gamma \in S_{-}(\A, \zeta_s)$ if and only if
$$\sum_{\substack{A \uplus B \uplus C = V(\Gamma)\\|B| = k}} \Gamma_A \otimes (\bar{\zeta_s} \zeta_s - \epsilon)(\Gamma_B) \otimes \Gamma_C = 0$$
for all $k$.
If $|B| = s+1$ for $s$ odd, then by Lemma~\ref{lem:dim} and Lemma~\ref{lem:simplex},
\begin{equation*}
(\bar{\zeta_s} \zeta_s - \epsilon)(\Gamma_B) =
\begin{cases}
-2 & \dim \Gamma_B = s,\\
0 & \mathrm{otherwise.}
\end{cases}
\end{equation*}
When $s$ is odd and $\dim \Gamma \geq s$
$$\sum_{\substack{A \uplus B \uplus C = V(\Gamma)\\|B| = s+1}} \Gamma_A \otimes (\bar{\zeta_s} \zeta_s - \epsilon)(\Gamma_B) \otimes \Gamma_C \ne 0$$
since the exists at least one subset $B \subseteq V(\Gamma)$ such that $|B| = s+1$ and $\dim \Gamma_B = s$.
Therefore $\Gamma \not\in S_{-}(\A, \zeta_s)$ when $s$ is odd and $\dim \Gamma \geq s$.
\end{proof}

We now provide two more lemmas which are useful in computing $\bar{\zeta_s} \zeta_s$ and hence useful in determining when a simplicial complex is in $S_{-}(\A,\zeta_s)$.

\begin{lem}
If $\Gamma$ is a simplicial complex such that $|V(\Gamma)|$ is odd, then $(\bar{\zeta_s} \zeta_s)(\Gamma) = \epsilon(\Gamma)$.
\label{lem:odd}
\end{lem}
\begin{proof}
First note if $|V(\Gamma)|$ is odd then $\Gamma \ne \varnothing$ and $\epsilon(\Gamma) = 0$.
Now,
\begin{align*}
(\bar{\zeta_s} \zeta_s)(\Gamma) &= \sum_{A \uplus B = V(\Gamma)} \bar{\zeta_s}(\Gamma_A) \zeta_s(\Gamma_B) \\
&= \sum_{\substack{A \uplus B = V(\Gamma) \\ \dim \Gamma_A < s \\ \dim \Gamma_B < s}} (-1)^{|A|} \\
&= 0
\end{align*}
since for $A \uplus B = V(\Gamma)$ the term $\bar{\zeta_s}(\Gamma_A) \zeta_s(\Gamma_B)$ will cancel the term $\bar{\zeta_s}(\Gamma_B) \zeta_s(\Gamma_A)$  as $|A|$ and $|B|$ have different parity.
\end{proof}

\begin{lem}
If $\Gamma$ is a simplicial complex and $x \in V(\Gamma)$ is not contained in any $s$-dimensional face of $\Gamma$, then  $(\bar{\zeta_s} \zeta_s)(\Gamma) = \epsilon(\Gamma)$.
\label{lem:face}
\end{lem}
\begin{proof}
First note if $x \in V(\Gamma)$ exists then $\Gamma \ne \varnothing$ so $\epsilon(\Gamma) = 0$.
We compute
\begin{align*}
(\bar{\zeta_s} \zeta_s)(\Gamma) &= \sum_{A \uplus B = V(\Gamma)} \bar{\zeta_s}(\Gamma_A) \zeta_s(\Gamma_B) \\
&=\sum_{A \uplus B = V(\Gamma) \setminus \{x\}} \bar{\zeta_s}(\Gamma_{A \cup \{x\}}) \zeta_s(\Gamma_B) + \sum_{A \uplus B = V(\Gamma) \setminus \{x\}} \bar{\zeta_s}(\Gamma_{A}) \zeta_s(\Gamma_{B \cup \{x\}})\\
&=\sum_{A \uplus B = V (\Gamma) \setminus \{x\}} (-1)^{|A| + 1}\zeta_s(\Gamma_{A \cup \{x\}}) \zeta_s(\Gamma_B) + \sum_{A \uplus B = V(\Gamma)  \setminus \{x\}} (-1)^{|A|}\zeta_s(\Gamma_{A}) \zeta_s(\Gamma_{B \cup \{x\}})\\
&= 0
\end{align*}
where we use that $\zeta_s(\Gamma_{A \cup \{x\}}) = \zeta_s(\Gamma_{A})$ because $x$ is not in any $s$-dimensional face of $\Gamma$.
\end{proof}

\section{$q$-analogues}\label{qanalogue}
In this section we develop a $q$-analogue of the characters $\zeta_s$ that we defined earlier.  This will in turn allow us to define a $q$-analogue of the $s$-chromatic symmetric function.  We will see that this $q$-analogue can distinguish an infinite family of graphs which the chromatic symmetric function cannot.  Additionally, we will discuss identities that can be obtained using principal specializations of this $q$-analogue of the chromatic symmetric function.

Recall that $\mathbb{K}$ is a field of characteristic 0.
Let $F=\mathbb{K}[q]$ be the  the polynomial ring in the variable $q$. For a graph $G$, the \emph{rank} of $G$, denoted by $rk(G)$ is the number of edges in a maximal subforest of $G$.  Given $s\geq 1$ and a simplicial complex $\Gamma$, define the map $\zeta_{s,q}:  \mathcal{A}\rightarrow F$ by $\zeta_{s,q}(\Gamma) = q^{rk(\Gamma^{(1)})}\zeta_s(\Gamma)$ and extend linearly.  Since the rank of the disjoint union of two graphs is the sum of their ranks, we get that  $\zeta_{s,q}$ is a character of $\A$ over $F$.

It is clear that $\zeta_{s,1} = \zeta_s$ and since  the only graphs with rank zero are those comprised of only isolated vertices, $\zeta_{s,0}=\zeta_1$.  These  remarks imply that  $\Psi_{\zeta_{s,q}}$ is the $s$-chromatic symmetric function when $q=1$ and is Stanley's chromatic symmetric function when $q=0$.

For a graph $G$ the value of $s$ in $\zeta_{s,q}$ is irrelevant.  In light of this and to simplify notation, we will use $\psi(G)$ for $\Psi_{\zeta_{1}}(G)$ and $\psi^q(G)$ for $\Psi_{\zeta_{2,q}}(G)$.
The choice $s=2$ is arbitrary since $\Psi_{\zeta_{2,q}}(G) = \Psi_{\zeta_{s,q}}(G)$ for any $s \geq 2$.
Note that applying equation~\ree{eq:Psi} for the character $\zeta_{2,q}$ on graphs  implies that
\begin{equation} \label{psiqEq}
\psi^q(G) = \sum_{V(G) = V_1\uplus V_2\uplus\cdots \uplus V_\ell} q^{\sum_i rk(G_{V_i})}M_{(|V_1|,|V_2|,\dots, |V_\ell|)}
\end{equation}
where the sum over all ordered set partitions of the vertex set of $G$.

\subsection{Unicyclic graphs}

A natural question to ask about Stanley's chromatic symmetric function, $\psi(G)$, is if it can distinguish between nonisomorphic graphs.  In \cite{s:sfgcpg} Stanley provided an example of two nonisomorphic graphs with the same  chromatic symmetric function.  Even though $\psi(G)$ cannot distinguish between nonisomorphic graphs, it is still an open problem to determine if it can distinguish between nonisomorphic trees.  Some  results in this direction can be found in~\cite{mmw:odtcs}.     

In~\cite{os:gecsf}, the authors described a way to write $\psi(G)$  as a linear combination of chromatic symmetric functions of  other graphs provided the original graph contains a triangle.   Using this, they showed how to construct an infinite family of pairs of nonisomorphic  graphs with the same chromatic symmetric function.  It was shown  in~\cite[Corollary 5]{mmw:odtcs} that one can recover the degree sequence of a tree using $\psi(G)$. However, this is not the case for unicyclic graphs (i.e. graphs with exactly one cycle) as was shown in~\cite{os:gecsf}.  In fact, in~\cite{os:gecsf} it is shown that the chromatic symmetric function cannot be used to determine the the number of leafs of a unicyclic graph.  It turns out that $\psi^q$ can be used to determine the number of leaves for unicyclic graphs as well as the number of vertices of degree two in the cycle.   After showing this, we explain how this gives  an infinite family of pairs of unicyclic graphs with the same chromatic symmetric function, but with different symmetric functions $\psi^q$. 

%
%
%
%
%

Following the notation in~\cite{os:gecsf}, for a unicyclic graph $G$, let $L_G$ be the number of leaves in $G$ and let $I_G$ be the number of vertices in $G$ with degree two which are contained in the cycle.

\begin{lem}\label{unicyLem}
Let $G$ be a connected unicyclic graph with $n$ vertices.  Then the coefficient of $q^{n-2}m_{(n-1,1)}$  in $\psi^q$ is $L_G+I_G$.
\end{lem}

\begin{proof}
By considering equation~\ree{psiqEq}, and noting that $M_{(n-1,1)}$ and $M_{(1,n-1)}$ will have the same coefficient in $\psi^q(G)$, one can see  that the coefficient $q^{n-2}m_{(n-1,1)}$ is the number of vertices $v$ such that $G\setminus v$ has rank $n-2$.  We will show that $G\setminus v$ has rank $n-2$ if and only if 
\begin{enumerate}
\item $v$ is a leaf or
\item  $\deg(v)=2$ and $v$ is in the cycle.
\end{enumerate}
Since $G$ is connected, this is equivalent to showing that $G\setminus v$ is connected if and only if $v$ is a leaf or $\deg(v)=2$ and $v$ is in the cycle.  

Suppose that $G\setminus v$ is connected.  If $v$ is not in the cycle, then $v$ must be a leaf since removing any other vertex of a tree disconnects the tree.  On the other hand, if $v$ is in the cycle, but has degree larger than two, then $v$ must be adjacent to a vertex, $w$ not in the cycle.  However, if $v$ is removed, it will disconnect $w$ from the rest of the graph.   It follows that the degree of $v$ must be two.

Now suppose that $v$ is  a leaf, then it is clear that $G\setminus v$ is connected.   On the other hand if  $v$ is in the cycle with degree two, then $v$ is only adjacent to vertices in the cycle. It follows that $G\setminus v$ is connected.  
\end{proof}

It was shown  in the proof of~\cite[Proposition 4.1]{os:gecsf} that if $G$ is a connected unicyclic graph with $n$ vertices such that the cycle has length $p$, then $(-1)^n[(p-1)L_G+I_G]$ is the coefficient of $p_{(n-1,1)}$ in the power sum expansion of $\psi(G)$.   Since $\psi(G)$ is obtained from $\psi^q(G)$ by setting $q=0$, this is also the coefficient of  $p_{(n-1,1)}$ in $\psi^q(G)$.  From Lemma~\ref{unicyLem}, we also know the coefficient of $q^{n-2}m_{(n-1,1)}$ in $\psi^q(G)$ is $L_G+I_G$.
This gives a system of linear equations of the form,
\begin{align*}
(p-1)L_G+I_G =& c_1\\
L_G+I_G = & c_2.
\end{align*}
Since $p$ is the length of a cycle, $p>2$ and so this system of linear equations has a unique solution. Thus we get the following proposition.
\begin{prop}\label{leafCycleProp}
Let $G$ be a connected unicyclic graph.  Then both $L_G$ and $I_G$ can be determined by  $\psi^q(G)$.  In particular, if $G$ and $H$ are connected and unicyclic such that $\psi^q(G) = \psi^q(H)$, then $L_G=L_H$ and $I_G=I_H$.
\end{prop}

\begin{figure}
\begin{center}
\begin{tikzpicture}[scale=.56] 
\node[draw,circle , fill=black!, scale=.5] (a) at (-2,2) {$$};
\node at (-2,2.5) {$u$};
\node (e) at (-4,2) {$T_1$};
\draw[dashed] (a)--(e);
\node[draw,circle , fill=black!, scale=.5] (b) at (2,2) {$$};
\node at (2,2.5) {$v$};
\node (f) at (4,2) {$T_2$};
\draw[dashed] (b)--(f);

\node[draw,circle , fill=black!, scale=.5] (c) at (-2,-2) {$$};
\node at (-2,-2.5) {$z$};
\node (g) at (-4,-2) {$T_1$};
\draw[dashed] (c)--(g);

\node[draw,circle , fill=black!, scale=.5] (d) at (2,-2) {$$};
\node at (2,-2.5) {$w$}; 
\node (h) at (4,-2) {$T_2$};
\draw[dashed] (d)--(h);

\draw (d)--(a)--(c)--(d)--(b);

\node[draw,circle , fill=black!, scale=.5] (a) at (-2+10,2) {$$};
\node at (-2+10,2.5) {$u$};
\node (e) at (-4+10,2) {$T_1$};
\draw[dashed] (a)--(e);
\node[draw,circle , fill=black!, scale=.5] (b) at (2+10,2) {$$};
\node at (2+10,2.5) {$v$};
\node (f) at (4+10,2) {$T_2$};
\draw[dashed] (b)--(f);

\node[draw,circle , fill=black!, scale=.5] (c) at (-2+10,-2) {$$};
\node at (-2+10,-2.5) {$z$};
\node (g) at (-4+10,-2) {$T_1$};
\draw[dashed] (c)--(g);

\node[draw,circle , fill=black!, scale=.5] (d) at (2+10,-2) {$$};
\node at (2+10,-2.5) {$w$}; 
\node (h) at (4+10,-2) {$T_2$};
\draw[dashed] (d)--(h);

\draw (a)--(c)--(d)--(b)--(c);

\end{tikzpicture}
\end{center}
\caption{The unicyclic graphs from~\cite{os:gecsf} with the same chromatic symmetric function}\label{unicycFig}
\end{figure}
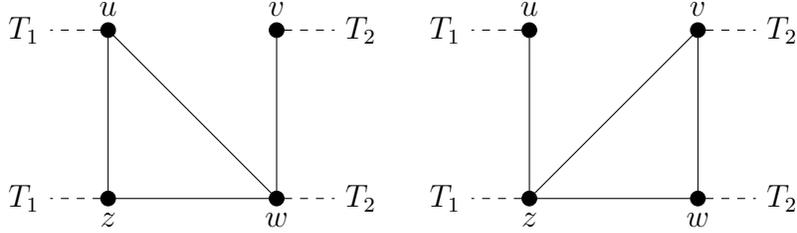

Now consider Figure~\ref{unicycFig}.  It is shown in~\cite{os:gecsf} that if $T_1$ and $T_2$ are nonisomorphic rooted trees which are attached by their root vertices, then the two graphs in the figure have the same chromatic symmetric function.  If we take exactly one of $T_1$ or $T_2$ to be the empty graph, 
 then the two graphs have a different number of leafs.  It follows that they can be distinguished by $\psi^q$ despite having the same chromatic symmetric function.

\subsection{Principal specializations and combinatorial identities}

For certain types of simplicial complexes,  the principal specialization of $\Psi_{s,-q}$ at $-1$ has a nice form which we will use to derive some identities involving compositions of integers.

First, we consider an identity that we can derive using the skeleton of a simplex.  Given $\alpha \vDash n$, we will use $\ell(\alpha)$ to denote the length of the composition $\alpha$. Moreover, we will use the notation  ${n\choose \alpha}$ for ${n \choose \alpha_1,\alpha_2,\dots, \alpha_{\ell(\alpha)}}$.  
\begin{prop}\label{completePrincSpec}
Let $\Gamma$ be a simplex on $n$ vertices.  As before, let $\Gamma^{(d)}$ be the $d$-skeleton of $\Gamma$.  
\begin{itemize}
\item[(a)] If $1\leq s\leq d$, then
$$
 ps^{1}(\Psi_{\zeta_{s,-q}}(\Gamma^{(d)}))(-1) =(-1)^n \sum_{\substack {(\alpha_1, \alpha_2,\dots, \alpha_\ell) \vDash n\\ \alpha_i\leq s \mbox{ } \forall i}} {n \choose \alpha} q^{n-\ell(\alpha)}
$$
where the sum is over all compositions $(\alpha_1, \alpha_2,\dots, \alpha_\ell)$  of $n$ such that $\alpha_i\leq s$ for all $i$.
\item [(b)] If $s>d\geq 1$, then 
$$
 ps^{1}(\Psi_{\zeta_{s,-q}}(\Gamma^{(d)}))(-1) =(-1)^n \sum_{\alpha \vDash n} {n \choose \alpha} q^{n-\ell(\alpha)}
$$
where the sum is over all compositions of $n$.
\end{itemize}
\end{prop}

\begin{proof}
We calculate the principal specialization directly.  Identify the vertex set of $\Gamma^{(d)}$ with $[n]$.  Since $s,d\geq 1$, the  one skeleton is a complete graph and so if $V_i$ is any nonempty subset of $[n]$, the rank of the one skeleton of $\Gamma^{(d)}|_{V_i}$ is $|V_i|-1$.  Therefore, equation~\ree{psiqEq} implies that
\begin{align*}
\Psi_{\zeta_{s,-q}}(\Gamma^{(d)})&= \sum_{(V_1,\dots, V _\ell)}  (-q)^{|V_1|-1}\cdots (-q)^{|V_\ell|-1} \zeta_s(\Gamma^{(d)}|_{V_1})\cdots  \zeta_s(\Gamma^{(d)}|_{V_\ell}) M_{(|V_1|,\dots, |V_\ell|)}\\
&= (-1)^n \sum_{(V_1,\dots, V_\ell)} (-1)^{\ell(\alpha)}q^{n-\ell(\alpha)} \zeta_s(\Gamma^{(d)}|_{V_1})\cdots  \zeta_s(\Gamma^{(d)}|_{V_\ell}) M_{(|V_1|,\dots, |V_\ell|)}.
\end{align*}
where the sum is over all ordered set partitions of $[n]$.

Now consider the induced subcomplex, $\Gamma^{(d)}|_{V_i}$.  If $s\leq d$, then 
$$
 \zeta_{s}(\Gamma^{(d)}|_{V_i}) =\begin{cases} 1 & \mbox{ if } |V_i|\leq s,\\ 0& \mbox{otherwise.} \end{cases}
$$
On the other hand, if $s>d$, then
$$
 \zeta_{s}(\Gamma^{(d)}|_{V_i})=1.
$$

Let $(V_1,\dots,V_\ell)$ be an ordered set partition of $[n]$ and denote by $\alpha=(\alpha_1,\dots,\alpha_\ell)$ the composition of $n$ obtained by putting $\alpha_1=|V_1|,\dots,\alpha_l=|V_\ell|$. We refer to $(V_1,\dots,V_\ell)$ as having \emph{type $\alpha$}.  There are 
$
{n \choose \alpha} 
$
ordered set partitions of $[n]$ with type $\alpha$.  The value of $  \zeta_s(\Gamma^{(d)}|_{V_1})\cdots \zeta_s(\Gamma^{(d)}|_{V_\ell}) M_{(|V_1|,\dots, |V_\ell|})$  depends on the type of $(V_1,\dots,V_\ell)$ and the value of $s$ and $d$.  In particular, if $s\leq d$ and $(V_1,\dots,V_\ell)$ has type  $(\alpha_1,\alpha_2,\dots, \alpha_\ell)$ then 
$$
\zeta_s(\Gamma^{(d)}|_{V_1})\cdots \zeta_s(\Gamma^{(d)}|_{V_\ell}) M_{(|V_1|,\dots, |V_\ell|}) =\begin{cases} M_{(\alpha_1,\dots, \alpha_\ell)} & \mbox{ if for all  $i$, } \alpha_i \leq s ,\\ 0 & \mbox{ otherwise.} \end{cases}
$$
It follows that  if $s\leq d$, 
$$
\Psi_{\zeta_{s,-q}}(\Gamma^{(d)}) = (-1)^n \sum_{\substack {(\alpha_1, \alpha_2,\dots, \alpha_\ell) \vDash n\\ \alpha_i\leq s \mbox{ } \forall i}}  (-1)^{\ell(\alpha)}{n \choose \alpha} q^{n-\ell(\alpha)} M_{\alpha}.
$$
Using a similar argument, if $s>d$, then 
$$
\Psi_{\zeta_{s,-q}}(\Gamma^{(d)}) =(-1)^n \sum_{\alpha \vDash n} (-1)^{\ell(\alpha)}{n \choose \alpha} q^{n-\ell(\alpha)} M_\alpha.
$$
Using the fact that $ps^1(M_{\alpha})(-1)=(-1)^{\ell(\alpha)}$ completes the proof.
\end{proof}

\begin{rem}
When $s>d$ in the previous proposition, the principal specialization is related to the $n^{th}$ Eulerian polynomial.  Let $ \mathfrak{S}_n$ be the symmetric group on $[n]$.  Given a permutation, $\omega = \omega_1\omega_2\cdots \omega_n\in  \mathfrak{S}_n$, define a descent of $\omega$ to be an index $i$ such that  $1\leq i<n$ and $\omega_{i}>\omega_{i+1}$.  Moreover, let $\mbox{des}(\omega)$ be the number of descents in $\omega$.  The \emph{$n^{th} $Eulerian polynomial} is given by
$$
A_n(q) = \sum_{\omega\in  \mathfrak{S}_n} q^{ \mbox{des}(\omega)}.
$$
Note that some authors (in particular Stanley~\cite{s:ec1}) define $A_n(q)$  with an exponent of $\mbox{des}(\omega) +1$ instead of $ \mbox{des}( \omega)$.  Although it is written slightly differently, exercise 1.133b in~\cite{s:ec1} shows that 
$$
A_n(q+1) =  \sum_{\alpha \vDash n} {n \choose \alpha} q^{n-\ell(\alpha)}.
$$
It follows from Proposition~\ref{completePrincSpec} part (b) that if $\Gamma$ is a simplex on $n$ vertices and $s>d$, then
$$
 ps^{1}(\Psi_{\zeta_{s,-q}}(\Gamma^{(d)}))(-1) =(-1)^n A_n(q+1).
$$

\end{rem}

Next we will look at specializations when the simplicial complex is a tree. Recall from equation~\ree{eval:minus1} that 
$$
(\zeta\circ S)(\Gamma)=  ps^{1}(\Psi_{\zeta}(\Gamma))(-1).
$$
Thus using Theorem \ref{anti:thm} yields
$$
 ps^{1}(\Psi_{\zeta_{s,q}}(\Gamma))(-1)  = \sum_{F\in \mathcal{F}(\Gamma^{(1)})}  (-1)^{c(F)} a(\Gamma^{(1)}/F)\zeta_{s,q}(\Gamma_{V,F}).
$$
If we instead consider $-q$, the previous equation shows that
\begin{align*}
 ps^{1}(\Psi_{\zeta_{s,-q}}(\Gamma))(-1)  &= \sum_{F\in \mathcal{F}(\Gamma^{(1)})}  (-1)^{c(F)} a(\Gamma^{(1)}/F)\zeta_{s,-q}(\Gamma_{V,F})\\
&=\sum_{F\in \mathcal{F}(\Gamma^{(1)})}  (-1)^{c(F)} a(\Gamma^{(1)}/F)(-1)^{rk(F)}\zeta_{s,q}(\Gamma_{V,F}).
\end{align*}
Since $|V(\Gamma)|= c(F) +rk(F)$ for any $F\in \mathcal{F}(\Gamma^{(1)})$ the following holds.
\begin{equation}\label{qPrincipalSpecialEq}
 ps^{1}(\Psi_{\zeta_{s,-q}}(\Gamma))(-1)  = (-1)^{|V(\Gamma)|}\sum_{F\in \mathcal{F}(\Gamma^{(1)})}  a(\Gamma^{(1)}/F)\zeta_{s,q}(\Gamma_{V,F})
\end{equation}

It turns out that for every tree on $n$ vertices, $\psi^q$ has the same principal specialization at $-1$.  In particular, we have the following.

\begin{prop}\label{treePrincProp}
Let $T_n$ be any tree with $n$ vertices and let $s>1$.  Then  the principal specialization at $-1$ is given by 
$$
 ps^{1}(\psi^{-q}(T_n))(-1)  = (-1)^n (q+2)^{n-1}.
$$
\end{prop}
\begin{proof}
First note that  every collection of edges of $T_n$ is a flat.  Moreover, when we contract  a flat with $k$ edges we get a tree with $n-k$ vertices.    Since $s>1$,  we have $\zeta_{s,q}(\Gamma_{V,F}) = q^{|F|}$. Therefore, equation~\ree{qPrincipalSpecialEq} gives
$$
 ps^{1}(\psi^{-q}(T_n))(-1)   = (-1)^n \sum_{F\subseteq E(T_n)}  a(T_n/F)q^{|F|}.
$$
 Since every subset of $E(T_n)$ is a flat and since the number of acyclic orientations of a tree with $m$ vertices is $2^{m-1}$,
$$
 ps^{1}(\psi^{-q}(T_n))(-1)   = (-1)^n \sum_{k=0}^{n-1}{n-1 \choose k} 2^{n-1-k}q^{k}
$$
which implies that
$$
 ps^{1}(\psi^{-q}(T_n))(-1)   = (-1)^n (q+2)^{n-1}.
$$
This completes the proof.
\end{proof}

We will now see how to derive an identity from the previous proposition. It will be useful to use a special type of tree to prove the result.  A \emph{star} with $n$ vertices is a tree with $n-1$ leaves and one central vertex which is adjacent to all other vertices.  It will be denoted by $St_n$.  We will also make use of \emph{falling factorials}.  Recall that the falling factorial is defined by $(x)_n = x(x-1)(x-2)\cdots (x-n+1)$.  Finally, given a composition $\alpha$ we will use the notation $\alpha^k$ for 
$$
\sum_{i=1}^{\ell(\alpha)} \alpha_i^k.
$$

\begin{cor}\label{starCor}
For all positive integers $k$ and $n$, we have
$$
(-1)^{n-1} \sum_{j=1}^k (-1)^j S(k,j) (n)_j = \sum_{\alpha \vDash n} (-1)^{\ell(\alpha)}{n \choose \alpha} \alpha^k
$$
where  $S(k,j)$ are the Stirling numbers of the second kind.
 \end{cor}
\begin{proof}\footnote{We thank the anonymous referee for the proof of Corollary~\ref{starCor}.}
  First, we calculate $ ps^{1}(\psi^{-q}(St_n))(-1)$ directly. Identify the vertex set of $St_n$ with $[n]$. Then for any ordered set  partition $(V_1,V_2,\dots, V_\ell)$  of $[n]$ of type $\alpha$, the induced subgraph with vertex set $V_i$ has rank $\alpha_i-1$ if the center vertex is in $V_i$ and 0 otherwise.  Considering all ordered set partitions of type $\alpha$ the number of times the center vertex appears in $V_i$ is given by $\frac{\alpha_i}{n}{n \choose \alpha}$.
Therefore,
$$
\psi^{-q}(St_n)=   \sum_{\alpha \vDash n} {n\choose \alpha}\left(\sum_{i=1}^{\ell(\alpha)} \frac{\alpha_i}{n} (-q)^{\alpha_i-1} M_{\alpha}\right).
$$
Again using that  $ps^1(M_{\alpha})(-1)=(-1)^{\ell(\alpha)}$, yields
$$
ps^{1}(\psi^{-q}(St_n))(-1) =   \sum_{\alpha \vDash n} {n\choose \alpha}\left(\sum_{i=1}^{\ell(\alpha)}(-1)^{\ell(\alpha)+\alpha_i-1} \frac{\alpha_i}{n} q^{\alpha_i-1} \right).
$$
Therefore using Proposition~\ref{treePrincProp}, we conclude that
$$
(-1)^n(q+2)^{n-1} =\sum_{\alpha \vDash n} {n\choose \alpha}\left(\sum_{i=1}^{\ell(\alpha)}(-1)^{\ell(\alpha)+\alpha_i-1} \frac{\alpha_i}{n} q^{\alpha_i-1} \right).
$$
Replacing $q$ by $-q$, multiplying both sides by $n$  and simplifying gives,
$$
-n(q-2)^{n-1} =\sum_{\alpha \vDash n}(-1)^{\ell(\alpha)} {n\choose \alpha}\left(\sum_{i=1}^{\ell(\alpha)}  \alpha_i q^{\alpha_i-1} \right).
$$
Now we apply the differential operator $\displaystyle \left(\frac{d}{dq}\right)^{j-1}$ to both sides of  the previous equation. This implies 
$$
 -(n)_j (q-2)^{n-j} = \sum_{\alpha \vDash n}(-1)^{\ell(\alpha)} {n\choose \alpha}\left(\sum_{i=1}^{\ell(\alpha)}   (\alpha_i)_j q^{\alpha_i-j}\right).
$$
Setting $q=1$ yields,
\begin{equation}\label{starEq}
(n)_j (-1)^{n-j-1} = \sum_{\alpha \vDash n}(-1)^{\ell(\alpha)} {n\choose \alpha}\left(\sum_{i=1}^{\ell(\alpha)} (\alpha_i)_j \right).
\end{equation}
Using the definition of $\alpha^k$,
$$
\sum_{\alpha \vDash n} (-1)^{\ell(\alpha)}{n \choose \alpha} \alpha^k=\sum_{\alpha \vDash n} (-1)^{\ell(\alpha)}{n \choose \alpha} \sum_{i=1}^{\ell(\alpha)} \alpha_i^k.
$$
Recalling the well known fact that if $k>0$,  $\alpha_i^k =\sum_{j=1}^k S(k,j) (\alpha_i)_j$, we see that 
$$
\sum_{\alpha \vDash n} (-1)^{\ell(\alpha)}{n \choose \alpha} \alpha^k = \sum_{\alpha \vDash n} (-1)^{\ell(\alpha)}{n \choose \alpha} \sum_{i=1}^{\ell(\alpha)} \left(\sum_{j=1}^k S(k,j) (\alpha_i)_j\right).
$$
Rearranging the summation, we have
$$
\sum_{\alpha \vDash n} (-1)^{\ell(\alpha)}{n \choose \alpha} \alpha^k =  \sum_{j=1}^k S(k,j) \sum_{\alpha \vDash n} (-1)^{\ell(\alpha)}{n \choose \alpha} \sum_{i=1}^{\ell(\alpha)} (\alpha_i)_j.
$$
Applying equation~\ree{starEq} we obtain the equation
$$
\sum_{\alpha \vDash n} (-1)^{\ell(\alpha)}{n \choose \alpha} \alpha^k  = \sum_{j=1}^k S(k,j)(n)_j (-1)^{n-j-1}.
$$
Factoring out the $(-1)^{n-1}$ on the right-hand side completes the proof.
\end{proof}

\section{Acknowledgments}

This paper originated during Fall 2014 in the Reading Combinatorics Seminar at Michigan State University with the active participation of S. Dahlberg and K. Barrese.
\nocite{*}
\bibliographystyle{alpha}

\bibliography{arxiv_v2_simplicial}

\end{document}